\documentclass[a4paper,12pt]{article}
\usepackage[latin1]{inputenc}
\usepackage{amssymb}
\usepackage{graphicx}
\usepackage{amsmath}
\usepackage{calc}
\usepackage{eurosym}
\usepackage{anysize}
\usepackage{amsthm}
\usepackage{enumerate}
\usepackage{semtrans}
\usepackage{eufrak}
\usepackage{mathrsfs}
\usepackage{color}
\usepackage{pst-node, pst-plot}
\usepackage{pstricks}
\usepackage{dsfont}
\marginsize{2,5cm}{2,5cm}{2,2cm}{2,3cm}

\title{Generalized Interacting Urn Models}
\author{Micka\"el Launay \and Vlada Limic}

\definecolor{jaune} {cmyk}{0,0.4,1,0}%
\definecolor{bleu} {cmyk}{1,0,0,0}%

\newcommand{\N}{\mathbb{N}}

\newcommand{\R}{\mathbb{R}}

\newcommand{\RUPa}{{IUM}}
\newcommand{\RUP}{{Interacting Urn Model}} 

\newcommand{\iy}{\infty}
\renewcommand{\P}{\mathbb{P}}
\newcommand{\E}{\mathbb{E}}
\newcommand{\FF}{{\mathcal F}}
\newcommand{\bck}{\!\!\!}
\newcommand{\be}{{\mathbf{e}}}

\newcommand{\indica}[1]{\mathbf{1}_{\{#1\}}}

\begin{document}

\newtheorem{lemme}{Lemma}[section]
\newtheorem{theorem}[lemme]{Theorem}
\newtheorem{prop}[lemme]{Proposition}
\newtheorem{coro}[lemme]{Corollaire}
\newtheorem{conj}[lemme]{Conjecture}
\newtheorem{defi}[lemme]{Definition}
\newtheorem{ex}[lemme]{Exemples}

\maketitle

\begin{abstract}
Interacting urns with exponential reinforcement were introduced and studied in Launay (2011). As its parameter $\rho$ tends to $\iy$, this reinforcement mechanism converges to the ``generalized'' reinforcement, in which the probability of draw may be $0$ or $1$ for some of the colors, depending on the current configuration. For a single urn, the generalized reinforcement is easy to analyse. We introduce and study the generalized interacting urn model with two or more urns and two colors. Our results concern the law of the so-called non-conformist urns, and answer in the asymptotic sense one of the open questions from the above mentioned paper.
\end{abstract}

\section{Introduction}

A common feature of reinforced processes is the definition of transition probabilities in terms of a given \textit{reinforcement weight sequence} $(w_i)_{i\in\N\cup\{0\}}\in(\mathbb{R}_+)^{\N\cup\{0\}}$.
At any time $n$, if a reinforced process $X_n$ can jump to one of the $\ell$ states $x_1$, $x_2$, ..., $x_\ell$ that have already been visited respectively $I_1, I_2, \ldots, I_\ell$ times, then for each $1\leq k\leq \ell$ the probability that it chooses $x_k$ is
$$\mathbb{P}\left[\left.X_{n+1}=x_k\right|X_0, X_1,\dots,X_n\right]=\frac{w_{I_k}}{w_{I_1}+w_{I_2}+\dots+w_{I_\ell}}.$$ 

Since the reinforcement weight sequence takes values in $\R_+$, we can note that the above probability is in $(0,1)$, or equivalently, that it never takes the values $0$ or $1$. In this paper we generalize the set of possible reinforcement sequences by introducing for $i,j\in\N\cup\{0\}$ the possibility for $w_i$ to be ``infinitely larger'' than $w_j$, that is allowing the values $0$ and $1$ for the above probability. 
In this framework, we prove a limit theorem for the {\RUP} with exponential reinforcement, introduced by Launay \cite{launay1}.

The simplest classical reinforced process consists of a single reinforced urn. 
Without further mention we will assume here and later that the initial time is $0$, and that the urn is empty at time $0$.
Suppose that there are $C$ possible colors $\mathfrak{c}_1, \mathfrak{c}_2, \ldots, \mathfrak{c}_C$.
At time $n$, let there be $N_k\geq 0$ balls of color $\mathfrak{c}_k$, for each $k=1,\ldots, C$.  Then at time $n+1$ a new ball of color $\mathfrak{c}_k$ is added to the urn with probability
$$\frac{w_{N_k}}{\displaystyle \sum_{k^\prime=1}^C w_{N_{k^\prime}}}.$$
The first such process was introduced by P\'{o}lya in 1930 (\cite{polya}), and since it has been generalized in a number of ways, see for example the survey by Pemantle \cite{pemantle}.

\begin{center}
	\pspicture(0,-.3)(14,2.4)

\psline[linecolor=gray]{-*}(1.7,.3)(2.5,1)
\psline[linecolor=gray]{-*}(1.7,2)(2.5,1.6)
			\pspolygon[fillstyle=solid,fillcolor=lightgray, linewidth=1pt, linecolor=lightgray, linearc=.2](.0,.0)(0,1.2)(.6,1.2)(.6,.6)(1.8,.6)(1.8,0)(0,0)
			\pspolygon[fillstyle=solid,fillcolor=white, linewidth=1pt, linecolor=gray, linearc=.2](0,1.25)(0,2.3)(1.8,2.3)(1.8,.7)(.65,.7)(.65,1.25)(0,1.25)

\rput[Bl](2.7,.9){4 black balls}
\rput[Bl](2.7,1.5){8 white balls}

\psline(-.1,-.1)(1.9,-.1)
\psline(-.1,-.1)(-.1,2.3)
\psline(1.9,-.1)(1.9,2.3)
			\pscircle[fillstyle=solid,fillcolor=black, linewidth=1pt,linecolor=black](.3,.3){.2}
			\pscircle[fillstyle=solid,fillcolor=black, linewidth=1pt,linecolor=black](.9,.3){.2}
			\pscircle[fillstyle=solid,fillcolor=black, linewidth=1pt,linecolor=black](1.5,.3){.2}
			\pscircle[fillstyle=solid,fillcolor=black, linewidth=1pt,linecolor=black](.3,.9){.2}
			\pscircle[fillstyle=solid,fillcolor=white, linewidth=1pt,linecolor=black](.9,1){.2}
			\pscircle[fillstyle=solid,fillcolor=white, linewidth=1pt,linecolor=black](1.5,1){.2}
			\pscircle[fillstyle=solid,fillcolor=white, linewidth=1pt,linecolor=black](.3,1.5){.2}
			\pscircle[fillstyle=solid,fillcolor=white, linewidth=1pt,linecolor=black](.9,1.5){.2}
			\pscircle[fillstyle=solid,fillcolor=white, linewidth=1pt,linecolor=black](1.5,1.5){.2}
			\pscircle[fillstyle=solid,fillcolor=white, linewidth=1pt,linecolor=black](.3,2){.2}
			\pscircle[fillstyle=solid,fillcolor=white, linewidth=1pt,linecolor=black](.9,2){.2}
			\pscircle[fillstyle=solid,fillcolor=white, linewidth=1pt,linecolor=black](1.5,2){.2}
			
\rput[Bl](5.5,1.7){{${\displaystyle \mathbb{P}[\text{a black ball is drawn in this urn}]=\frac{w_4}{w_8+w_4}}$}}
\rput[Bl](5.5,.6){{${\displaystyle \mathbb{P}[\text{a white ball is drawn in this urn}]=\frac{w_8}{w_8+w_4}}$}}
	\endpspicture
	\end{center}

\noindent	
{\footnotesize Figure 1:
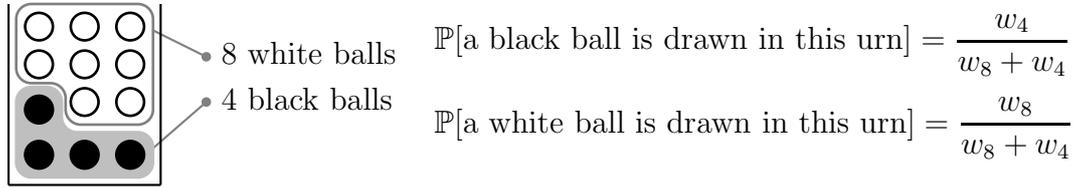
Transition probabilities for a particular realization of a single $w$-urn at time 12.}
\setcounter{figure}{1}

\vspace{0.3cm}
\noindent

As already mention, the first author recently introduced a multiple urn model, called the {\RUP} (IUM), in which the urns interact through shared memory. Suppose that a number $U\in\N$ of urns is given, and let us fix the interaction parameter $p\in [0,1]$. The dynamics is given as follows:
at each time $n$, each urn draws independently a ball either from itself with probability $1-p$, or from all the urns combined with probability $p$. In other words, the higher the $p$ the more memory is shared between the urns. To be more precise, suppose that at time $n$, for each $k=1,\ldots, C$ and $u=1,\ldots, U$, there is $N^u_k$ balls of color $\mathfrak{c}_k$ in the urn number $u$. Then for each $u$, at time $n+1$ we add a new ball of color $\mathfrak{c}_k$ in the urn $u$ with probability
$$(1-p)\frac{w_{N_k^u}}{\displaystyle \sum_{k^\prime=1}^C w_{N_{k^\prime}^u}}+p\frac{w_{N_k^*}}{\displaystyle \sum_{k^\prime=1}^C w_{N_{k^\prime}^*}},$$
where $N_k^*=\sum_{1\leq v\leq U}N_{k^\prime}^v$ is the total number of balls of color $\mathfrak{c}_k$ in the configuration at time $n$.

\begin{center}

\pspicture(0,-.6)(16,2.5)
\scalebox{0.75}{
\psline(-0.1,0)(1.6,0)
\psline(-0.1,0)(-0.1,2)
\psline(1.6,0)(1.6,2)
			\pscircle[fillstyle=solid,fillcolor=black, linewidth=1pt,linecolor=black](.25,.3){.23}
			\pscircle[fillstyle=solid,fillcolor=black, linewidth=1pt,linecolor=black](.75,.3){.23}
			\pscircle[fillstyle=solid,fillcolor=black, linewidth=1pt,linecolor=black](1.25,.3){.23}
			\pscircle[fillstyle=solid,fillcolor=black, linewidth=1pt,linecolor=black](.25,.8){.23}
			\pscircle[fillstyle=solid,fillcolor=white, linewidth=1pt,linecolor=black](.75,.8){.23}
			\pscircle[fillstyle=solid,fillcolor=white, linewidth=1pt,linecolor=black](1.25,.8){.23}
			\pscircle[fillstyle=solid,fillcolor=white, linewidth=1pt,linecolor=black](.25,1.3){.23}
			\pscircle[fillstyle=solid,fillcolor=white, linewidth=1pt,linecolor=black](.75,1.3){.23}
			\pscircle[fillstyle=solid,fillcolor=white, linewidth=1pt,linecolor=black](1.25,1.3){.23}
			\pscircle[fillstyle=solid,fillcolor=white, linewidth=1pt,linecolor=black](.25,1.8){.23}
			\pscircle[fillstyle=solid,fillcolor=white, linewidth=1pt,linecolor=black](.75,1.8){.23}
			\pscircle[fillstyle=solid,fillcolor=white, linewidth=1pt,linecolor=black](1.25,1.8){.23}

\psline(1.9,0)(3.6,0)
\psline(1.9,0)(1.9,2)
\psline(3.6,0)(3.6,2)
			\pscircle[fillstyle=solid,fillcolor=black, linewidth=1pt,linecolor=black](2.25,.3){.23}
			\pscircle[fillstyle=solid,fillcolor=white, linewidth=1pt,linecolor=black](2.75,.3){.23}
			\pscircle[fillstyle=solid,fillcolor=white, linewidth=1pt,linecolor=black](3.25,.3){.23}
			\pscircle[fillstyle=solid,fillcolor=white, linewidth=1pt,linecolor=black](2.25,.8){.23}
			\pscircle[fillstyle=solid,fillcolor=white, linewidth=1pt,linecolor=black](2.75,.8){.23}
			\pscircle[fillstyle=solid,fillcolor=white, linewidth=1pt,linecolor=black](3.25,.8){.23}
			\pscircle[fillstyle=solid,fillcolor=white, linewidth=1pt,linecolor=black](2.25,1.3){.23}
			\pscircle[fillstyle=solid,fillcolor=white, linewidth=1pt,linecolor=black](2.75,1.3){.23}
			\pscircle[fillstyle=solid,fillcolor=white, linewidth=1pt,linecolor=black](3.25,1.3){.23}
			\pscircle[fillstyle=solid,fillcolor=white, linewidth=1pt,linecolor=black](2.25,1.8){.23}
			\pscircle[fillstyle=solid,fillcolor=white, linewidth=1pt,linecolor=black](2.75,1.8){.23}
			\pscircle[fillstyle=solid,fillcolor=white, linewidth=1pt,linecolor=black](3.25,1.8){.23}

\psline(3.9,0)(5.6,0)
\psline(3.9,0)(3.9,2)
\psline(5.6,0)(5.6,2)
			\pscircle[fillstyle=solid,fillcolor=black, linewidth=1pt,linecolor=black](4.25,.3){.23}
			\pscircle[fillstyle=solid,fillcolor=black, linewidth=1pt,linecolor=black](4.75,.3){.23}
			\pscircle[fillstyle=solid,fillcolor=black, linewidth=1pt,linecolor=black](5.25,.3){.23}
			\pscircle[fillstyle=solid,fillcolor=black, linewidth=1pt,linecolor=black](4.25,.8){.23}
			\pscircle[fillstyle=solid,fillcolor=black, linewidth=1pt,linecolor=black](4.75,.8){.23}
			\pscircle[fillstyle=solid,fillcolor=black, linewidth=1pt,linecolor=black](5.25,.8){.23}
			\pscircle[fillstyle=solid,fillcolor=black, linewidth=1pt,linecolor=black](4.25,1.3){.23}
			\pscircle[fillstyle=solid,fillcolor=black, linewidth=1pt,linecolor=black](4.75,1.3){.23}
			\pscircle[fillstyle=solid,fillcolor=white, linewidth=1pt,linecolor=black](5.25,1.3){.23}
			\pscircle[fillstyle=solid,fillcolor=white, linewidth=1pt,linecolor=black](4.25,1.8){.23}
			\pscircle[fillstyle=solid,fillcolor=white, linewidth=1pt,linecolor=black](4.75,1.8){.23}
			\pscircle[fillstyle=solid,fillcolor=white, linewidth=1pt,linecolor=black](5.25,1.8){.23}			
\rput[Bl](0.3,-.5){urn 1}
\rput[Bl](2.3,-.5){urn 2}
\rput[Bl](4.2,-.5){urn 3}
}			
\rput[Bl](4.9,2){{${\displaystyle \mathbb{P}[\text{urn 1 draws a black ball}]=(1-p)\frac{w_4}{w_8+w_4}+p\frac{w_{13}}{w_{23}+w_{13}}},$}}
\rput[Bl](4.9,.9){{${\displaystyle \mathbb{P}[\text{urn 2 draws a black ball}]=(1-p)\frac{w_1}{w_{11}+w_1}+p\frac{w_{13}}{w_{23}+w_{13}}},$}}
\rput[Bl](4.9,-.2){{${\displaystyle \mathbb{P}[\text{urn 2 draws a black ball}]=(1-p)\frac{w_8}{w_4+w_8}+p\frac{w_{13}}{w_{23}+w_{13}}}.$}}

\endpspicture
\end{center}

\noindent
{\footnotesize Figure 2:
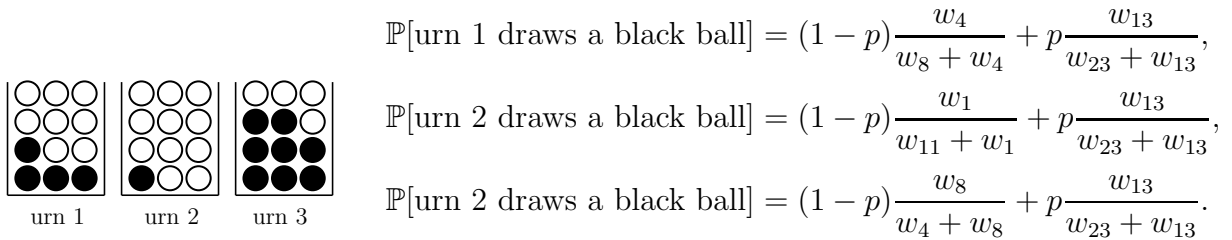
Transition probabilities for a particular realization of the {\RUPa} with $U=3$ at time 12.}
\setcounter{figure}{2}

\vspace{0.3cm}
\noindent
For this model, it is proved in \cite{launay1} that if 
$\liminf_{i\to\iy}w_{i+1}/w_{i}>1$, and in particular if $w_i=\rho^i$ with $\rho>1$, then
\begin{itemize}
\item if $p\geq 1/2$ all the $U$ urns eventually fixate on the same color ;
\item if $p < 1/2$ there exist two colors $\mathfrak{c}_1$ and $\mathfrak{c}_2$ such that a
certain number 
of urns eventually draw a ball of color $\mathfrak{c}_1$ each time they draw out of all the urns combined and draw a ball of color $\mathfrak{c}_2$ otherwise, while all the 
other urns fixate on $\mathfrak{c_1}$. 
The number $N$ of urns that keep drawing balls of both colors is random and observes a deterministic bound 
$N < U/(2-2p)$, almost surely (see \cite{launay1}).
\end{itemize}
An open question from \cite{launay1} was to determine the law of $N$. 
The generalized reinforcement weights, will allow us to answer this question (see Theorem \ref{theorem} below) in the asymptotic regime $\rho \to \iy$, provided that $U=2$ or that $U$ is an odd number.

So let us now define the set of {\it generalized reinforcement weight sequence} as
$$\mathcal{W}:=\left(\R_+\times\R\right)^{\N\cup\{0\}}.$$
For $w=(w_i)_{i\in\N\cup\{0\}}\equiv (u_i,v_i)_{i\in\N\cup\{0\}}\in\mathcal{W}$ in the following, we shall use the abbreviated formal notation:
$$w_i=(u_i,v_i)=:u_i\iy^{v_i},~~\forall i\in\N\cup\{0\}.$$
The symbol $\iy$ in this notation is meant to recall the reader of the approximation 
$\rho^i \approx \iy^i$ in the sense of Theorem \ref{T:approx} below.

If $a\in\N\cup\{0\}$ and $B=(b_1,b_2,\dots,b_\ell)\in \bigcup_{m\in\N}{\left(\N\cup\{0\}\right)}^m$ then we define
\begin{equation*}
\pi_{w}(a,B)=\pi_{w}(a,b_1,b_2,\dots,b_\ell)=\left\{
\begin{array}{ll}
0&\qquad \displaystyle \text{if}\quad v_a<\max_{1\leq k\leq \ell}{v_b};\\
\displaystyle \frac{u_a}{\displaystyle u_a+\sum_{1\leq k\leq \ell}u_{b_k}\, \delta_{v_a}(v_{b_k})}&\qquad \text{if}\quad \displaystyle v_a=\max_{1\leq k\leq \ell}{v_b};\\
1&\qquad \text{if}\quad \displaystyle v_a>\max_{1\leq k\leq \ell}{v_b}.
\end{array}
\right.
\end{equation*}
This quantity should be interpreted as the probability to draw a black ball in an urn that contains $a$ black balls and $b_k$ balls of color $\mathfrak{c}_k$ (different of black) for each $k\in\{1,\dots,\ell\}$. In particular the probability to draw a black ball among $a$ black and $b$ white balls is:
\begin{equation*}
\pi_{w}(a,b):=\left\{
\begin{array}{ll}
0&\qquad \text{if}\quad v_a<v_b;\\
\frac{u_a}{u_a+u_b}&\qquad \text{if}\quad v_a=v_b;\\
1&\qquad \text{if}\quad v_a>v_b.
\end{array}
\right.
\end{equation*}
When clear from the context we will simply write $\pi(a,b)=\pi_{w}(a,b)$. 

\textbf{Remark} Note that a classical reinforcement weight sequence $(w_i)_{i\in\N\cup\{0\}}\in(\R_+)^{\N\cup\{0\}}$, could be identified with a generalized reinforcement weight sequence $(u_i\iy^{v_i})_{i\in\N\cup\{0\}}$ where $u_i=w_i$ and $v_i$ is constant.

For reasons of simplicity, we will consider from now on the above urn processes with only two colors, black and white. In the final section we make some remarks about the setting of three or more colors.

Let us first state some easy to derive facts on the behaviour of a single urn that corresponds to the above generalized reinforcement weight sequence.
\begin{lemme}
\label{L:single urn}
(a) If $v_i$ is increasing, the color of the first ball is chosen uniformly at random, and then all the balls drawn have the same color almost surely.\\
(b) If $v_i$ is decreasing, the color of the ball drawn at even times is chosen uniformly at random, and for each $k\in \N$ the color drawn at time $2k+1$ is almost surely different from that drawn at time $2k$.\\
(c) For each $i_0$ such that $v_{i_0}<v_i$ for all $i<i_0$, there is almost surely $i_0$ black balls and $i_0$ white balls in the urn at time $2i_0$.
\end{lemme}
\begin{proof}
Statements (a)--(b) are easy and left to the reader.
For (c) denote by $\tau_{i_0}$ the first (random) time at which one of the colors is drawn
exactly $i_0$ times. 
Recall that the urn is empty at the initial time $0$.
Suppose WLOG that this color is black. Then clearly $\P(i_0\leq \tau_{i_0}\leq 2i_0-1) =1$ and the number of white balls drawn at time $\tau_{i_0}$ is precisely 
$\tau_{i_0}-i_0 <i_0$.
Due to the above assumption on the sequence $(v_i)_i$, at time
$\tau_{i_0}$ the urn is bound to draw white, and 
similarly at each time $\{\tau_{i_0}+1,\ldots, 2i_0-1\}$ 
the urn will draw white ball almost surely, which implies the stated claim.
\end{proof}

Davis \cite{davis} proved the well-known fact: 
a single urn with a classical reinforcement weight sequence 
$(w_i)_{i\in\N\cup\{0\}}\in (\R_+)^{\N\cup\{0\}}$ eventually fixates 
on a single color if and only if $\sum_{i=0}^\iy w_i^{-1}<\iy$. 
The following proposition extends this result in the setting of generalized reinforcement weight sequences.

\begin{prop} If the sequence $(v_i)_{i\in\N\cup\{0\}}$ attains its minimum $\underline{v}:=\min_{i\in\N\cup\{0\}}v_i$ and
$$\sum_{i=0}^\iy\frac{\indica{v_i=\underline{v}}}{u_i}<\iy,$$
then the urn eventually fixates on a single color almost surely. 
Otherwise, both black and white balls are drawn infinitely often almost surely.
\end{prop}

\begin{proof}
First note that if $(v_i)_{i\in\N\cup\{0\}}$ does not attain its infimum, 
then there are infinitely many indices $j\in\N$ such that $v_j<v_i$ for all $i<j$. 
Then using Lemma \ref{L:single urn}(c), for each such $j$, there are $j$ balls of each color in the urn at time $2j$. 
We can therefore conclude that almost surely, balls of both colors are drawn infinitely often.

On the contrary, if $(v_i)_{i\in\N\cup\{0\}}$ attains its minimum, denote by $\alpha(0)$, $\alpha(1)$, $\alpha(2)$... the (finite or infinite) sequence of subscripts such that $v_{\alpha(i)}=\underline{v}:=\min_{j\in\N\cup\{0\}}v_j$, for each $i$. 
Let us suppose first that this sequence is infinite.
We can then use the classical time line argument of Rubin.
\begin{center}
	\pspicture(0,-.6)(10,1.6)

\psline{|-}(0,0)(1,0)
\psline{|-}(1,0)(3,0)
\psline{|-}(3,0)(3.7,0)
\psline{|-}(3.7,0)(7,0)
\psline{|-}(7,0)(9.1,0)
\psline{|->}(9.1,0)(10,0)
\rput(.5,.3){{$\mathcal{E}^\prime(0)$}}
\rput(2,.3){{$\mathcal{E}^\prime(1)$}}
\rput(3.35,.3){{$\mathcal{E}^\prime(2)$}}
\rput(5.35,.3){{$\mathcal{E}^\prime(3)$}}
\rput(8.05,.3){{$\mathcal{E}^\prime(4)$}}

\psline{|-}(0,1)(.5,1)
\psline{|-}(.5,1)(5.1,1)
\psline{|-}(5.1,1)(6,1)
\psline{|-}(6,1)(7.8,1)
\psline{|->}(7.8,1)(10,1)
\rput(.25,1.3){{$\mathcal{E}(0)$}}
\rput(2.8,1.3){{$\mathcal{E}(1)$}}
\rput(5.55,1.3){{$\mathcal{E}(2)$}}
\rput(6.9,1.3){{$\mathcal{E}(3)$}}

\rput[Bl](-1.4,1){urn 1}
\rput[Bl](-1.4,0){urn 2}
	\endpspicture
	\end{center}

\noindent
{\footnotesize Figure 3:
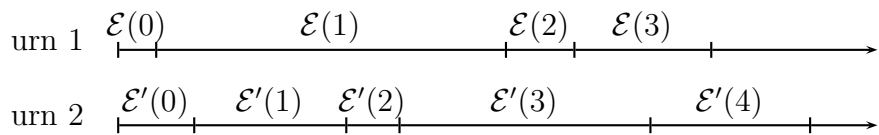
Rubin's time line technique}
\setcounter{figure}{3}

\vspace{0.3cm}
\noindent
The time interval indicated by $\mathcal{E}(i)$ (or $\mathcal{E}'(i)$) in the figure has  length distributed as Exponential (rate $u_{\alpha(i)}$) random variable, for each $i\in \N$. Moreover, all the Exponential variables are mutually independent. 
Note that by Lemma \ref{L:single urn}(c), at time $2\alpha(0)$ there are almost surely $\alpha(0)$ black balls and $\alpha(0)$ white balls in the urn.
Then one can construct a realization of the urn process from the time-lines just as in Davis \cite{davis} (see e.g.~also \cite{attracom}) using the fact that if the current count of one color is at $i$ such that $v_i=\underline{v}$, and of the other color is at $j$ such that $v_j>\underline{v}$, then the urn ``deterministically'' draws the latter color on this time step. The time-lines are used to decide the next draw on steps where both counts are contained in $(\alpha(i))_{i\geq 0}$, that is when the draw is not deterministic.
	
If on the other hand, the sequence $\alpha(i)$ defined above is finite, let $\alpha^*$ be the maximum of $\alpha(i)$. Then the urn will find itself in the situation analogous to that of Lemma \ref{L:single urn}(a) after finitely many steps, in the sense that 
the color for which the count becomes first higher than $\alpha^*$ will be drawn forever after.
\end{proof}

Let us denote by $\mathfrak{U}:=\{\mathfrak{b},\mathfrak{w}\}^\N$ the set of all possible evolutions of an urn. If $\mathfrak{u}\in\mathfrak{U}$, then for $i\in\N$, $\mathfrak{u}(i)$ equals $\mathfrak{b}$ (resp. $\mathfrak{w}$) if the ball drawn at time $i$ is black (resp. white). Then $\mathfrak{U}^2$ is the set of all possible joint evolutions of the two urns. Let $F\subset \mathfrak{U}^2$ be the set of all evolutions of two urns such that both of them fixate on the same color, or more precisely;
$$F:=\{(\mathfrak{u},\mathfrak{v})\in\mathfrak{U}^2,~ \text{ such that }~\exists n_0,\forall n\geq n_0, \mathfrak{u}(n)=\mathfrak{v}(n)=\mathfrak{u}(n_0)\}$$

Let us now state the first main result of this paper.
From now on we consider the Generalized Interacting Urn Model (GIUM) with
two urns (that is, $U=2$). In the final section we address extensions to three or more urns.
\begin{theorem}\label{theorem} If $w_i=\iy^i$ (or equivalently $w_i=(1,i)$)
and if the interaction parameter $p\leq 1/2$, then
$$\mathbb{P}\left[F\right]=\frac{1-p+C(p)(3p-2)}{(1-p)(2-\lambda_{-}(p))}+C(p)$$
where
$$C(p)=-\frac{2 (1-p)^2 (1+p)}{2 p^3-6p^2+9p-4}\text{~~~and~~~}\lambda_{-}(p)=\frac{1-\sqrt{1-p^2\left(1-\frac{p}{2}\right)^2}}{2\left(1-p/2\right)^2}.$$
\end{theorem}

\textbf{Remark} The reinforcement weight sequence $w_i=\iy^i$ could be replaced 
by any $w_i=u_i\iy^{v_i}$ with $u_i=1$ and $(v_i)_i$ increasing. 
The law of the GIUM is trivially equal to that of GIUM with $w_i=\iy^i$.
Indeed, this law is entirely determined by the family $(\pi_w(i,j))_{i,j})$ and when $(v_i)_i$ is increasing, we have $\pi_w(i,j)=\indica{i>j}+\indica{i=j}/2$.

Note that if $p>1/2$, we already know (Theorem 3.2 from \cite{launay1}) that the two urns fixate on the same color almost surely. Therefore, the graph of
$$p\mapsto \mathbb{P}\left[F\right]$$
is as follows:

\begin{center}
	\pspicture(-.5,-.5)(4,4)
\scalebox{0.75}{
	\psaxes[labels=none,Dx=0.5,Dy=0.5]{->}(0,0)(5.5,5.5)

\pscurve(0,2.5)(.1,2.5261)(.2,2.5547)(.3,2.5858)(.4,2.6197)(.5,2.6565)(.6,2.6965)(.7,2.74)(.8,2.7873)(.9,2.8387)(1,2.8946)(1.1,2.9555)(1.2,3.0218)(1.3,3.0943)(1.4,3.1737)(1.5,3.2607)(1.6,3.3566)(1.7,3.4625)(1.8,3.5802)(1.9,3.7114)(2,3.8588)(2.1,4.0257)(2.2,4.2162)(2.3,4.4363)(2.4,4.6938)(2.5,5)
	\psline(2.5,5)(5,5)	
	\rput[Bl](-.5,-.5){0}
	\rput[Bl](-.8,.4){0,1}
	\rput[Bl](-.8,.9){0,2}
	\rput[Bl](-.8,1.4){0,3}
	\rput[Bl](-.8,1.9){0,4}
	\rput[Bl](-.8,2.4){0,5}
	\rput[Bl](-.8,2.9){0,6}
	\rput[Bl](-.8,3.4){0,7}
	\rput[Bl](-.8,3.9){0,8}
	\rput[Bl](-.8,4.4){0,9}
	\rput[Bl](-.8,4.9){1}
	\rput(1,-.4){0.2}
	\rput(2,-.4){0.4}
	\rput(3,-.4){0.6}
	\rput(4,-.4){0.8}
	\rput(5,-.4){1}
}
	\endpspicture

\noindent
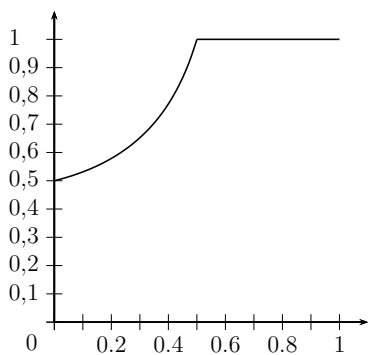
{\footnotesize Figure 4: Fixation probability plot}
\setcounter{figure}{4}

\vspace{0.3cm}
\noindent
	\end{center}

Before proving Theorem \ref{theorem} in Section \ref{S:proofs}, we explain 
our main motivation for obtaining this result, which concerns the limiting regime for the Interacting Urn with exponential reinforcement, where $\rho \to \iy$.

For any given $\rho>1$ set $w(\rho)=(\rho^i)_{i\in\N\cup\{0\}}$, let and $w(\iy)=(\iy^i)_{i\in\N\cup\{0\}}$. Denote by $\P_{\rho},\E_\rho$ (resp.~$\P_{\iy},E_{\iy}$) the probability laws and expectations induced by the reinforcement weight $w(\rho)$ (resp.~$w(\iy)$) on $\mathfrak{U}^2$.
In fact, in Section \ref{S:approx} we will enrich the filtration to also include the i.i.d.~Bernoulli coin flips which tell us from which drawing pool (the urn itself or the two urns combined) are the draws made, and we will call the resulting laws again $\P_\rho,\P_\iy$. This should however not confuse the reader since the projection of the enriched laws to $\mathfrak{U}^2$ gives precisely the above defined laws. 
\begin{theorem}\label{T:approx}
We have
$$\lim_{\rho\to\iy}\P_{\rho}\left[F\right]=\P_{\iy}\left[F\right],$$
where the right-hand side equals the expression given in Theorem \ref{theorem}.
\end{theorem}

Note that $F$ is a countable set, so we have
$$\P_{\rho}\left[F\right]=\P_{\rho}\left[\bigcup_{(\mathfrak{u}, \mathfrak{v})\in F}\{(\mathfrak{u},\mathfrak{v})\}\right]=\sum_{(\mathfrak{u},\mathfrak{v})\in F}\P_{\rho}\left[(\mathfrak{u},\mathfrak{v})\right].$$
For any $(\mathfrak{u},\mathfrak{v})\in F$, it is relatively easy to show that
$$\lim_{\rho\to\iy}\P_{\rho}\left[(\mathfrak{u},\mathfrak{v})\right]=\P_{\iy}\left[(\mathfrak{u},\mathfrak{v})\right].$$
However, this convergence is not sufficient to prove the convergence of the series, and the most concise rigorous argument that we could find is perhaps surprisingly long. Since we use some elements of the proof of Theorem \ref{theorem} in the proof of Theorem \ref{T:approx}, this is postponed to Section \ref{S:approx}.

\section{Proofs of Theorems \ref{theorem} and \ref{T:approx}}
\label{S:proofs}
Our first observation is that for all $i,j\in\N$, we have $\pi(i,j)=\pi(i-1,j-1)$. This means that if an urn contains both black and white balls, one can withdraw one black ball and one white ball from this urn without changing the law of the next ball to be drawn. In other words, if an urn contains $B$ black balls and $W$ white balls with $B\geq W$, its behaviour is the same as the behaviour of an urn that contains $B-W$ black balls and no white balls.

This remark allows us to simplify the set of possible configurations for two interacting urns. Let us consider two urns, containing respectively $B_1$ and $B_2$ black balls and $W_1$ and $W_2$ white balls. We can classify the possible configurations into three cases: 

\textbf{Configuration 1.} $B_1+B_2=W_1+W_2$. It also means that $B_1-W_1=W_2-B_2$, and according to the previous remark, this configuration is equivalent to a configuration in which one urn contains $|B_1-W_1|=|B_2-W_2|$ black balls and the other as many white balls. We denote this configuration by $C_1(|B_1-W_1|)$.

\begin{center}
	\pspicture(0,-.6)(13.6,2.5)

\psline(-0.1,0)(1.6,0)
\psline(-0.1,0)(-0.1,2)
\psline(1.6,0)(1.6,2)
			\pscircle[fillstyle=solid,fillcolor=black, linewidth=1pt,linecolor=black](.25,.3){.23}
			\pscircle[fillstyle=solid,fillcolor=black, linewidth=1pt,linecolor=black](.75,.3){.23}
			\pscircle[fillstyle=solid,fillcolor=black, linewidth=1pt,linecolor=black](1.25,.3){.23}
			\pscircle[fillstyle=solid,fillcolor=black, linewidth=1pt,linecolor=black](.25,.8){.23}
			\pscircle[fillstyle=solid,fillcolor=white, linewidth=1pt,linecolor=black](.75,.8){.23}
			\pscircle[fillstyle=solid,fillcolor=white, linewidth=1pt,linecolor=black](1.25,.8){.23}
			\pscircle[fillstyle=solid,fillcolor=white, linewidth=1pt,linecolor=black](.25,1.3){.23}
			\pscircle[fillstyle=solid,fillcolor=white, linewidth=1pt,linecolor=black](.75,1.3){.23}
			\pscircle[fillstyle=solid,fillcolor=white, linewidth=1pt,linecolor=black](1.25,1.3){.23}
			\pscircle[fillstyle=solid,fillcolor=white, linewidth=1pt,linecolor=black](.25,1.8){.23}
			\pscircle[fillstyle=solid,fillcolor=white, linewidth=1pt,linecolor=black](.75,1.8){.23}
			\pscircle[fillstyle=solid,fillcolor=white, linewidth=1pt,linecolor=black](1.25,1.8){.23}

\psline(1.9,0)(3.6,0)
\psline(1.9,0)(1.9,2)
\psline(3.6,0)(3.6,2)
			\pscircle[fillstyle=solid,fillcolor=black, linewidth=1pt,linecolor=black](2.25,.3){.23}
			\pscircle[fillstyle=solid,fillcolor=black, linewidth=1pt,linecolor=black](2.75,.3){.23}
			\pscircle[fillstyle=solid,fillcolor=black, linewidth=1pt,linecolor=black](3.25,.3){.23}
			\pscircle[fillstyle=solid,fillcolor=black, linewidth=1pt,linecolor=black](2.25,.8){.23}
			\pscircle[fillstyle=solid,fillcolor=black, linewidth=1pt,linecolor=black](2.75,.8){.23}
			\pscircle[fillstyle=solid,fillcolor=black, linewidth=1pt,linecolor=black](3.25,.8){.23}
			\pscircle[fillstyle=solid,fillcolor=black, linewidth=1pt,linecolor=black](2.25,1.3){.23}
			\pscircle[fillstyle=solid,fillcolor=black, linewidth=1pt,linecolor=black](2.75,1.3){.23}
			\pscircle[fillstyle=solid,fillcolor=white, linewidth=1pt,linecolor=black](3.25,1.3){.23}
			\pscircle[fillstyle=solid,fillcolor=white, linewidth=1pt,linecolor=black](2.25,1.8){.23}
			\pscircle[fillstyle=solid,fillcolor=white, linewidth=1pt,linecolor=black](2.75,1.8){.23}
			\pscircle[fillstyle=solid,fillcolor=white, linewidth=1pt,linecolor=black](3.25,1.8){.23}

\psline(4.9,0)(6.6,0)
\psline(4.9,0)(4.9,2)
\psline(6.6,0)(6.6,2)
			\pscircle[fillstyle=solid,fillcolor=black, linewidth=1pt,linecolor=black](5.25,.3){.23}
			\pscircle[fillstyle=solid,fillcolor=black, linewidth=1pt,linecolor=black](5.75,.3){.23}
			\pscircle[fillstyle=solid,fillcolor=white, linewidth=1pt,linecolor=black](6.25,.3){.23}
			\pscircle[fillstyle=solid,fillcolor=white, linewidth=1pt,linecolor=black](5.25,.8){.23}
			\pscircle[fillstyle=solid,fillcolor=white, linewidth=1pt,linecolor=black](5.75,.8){.23}
			\pscircle[fillstyle=solid,fillcolor=white, linewidth=1pt,linecolor=black](6.25,.8){.23}
			\pscircle[fillstyle=solid,fillcolor=white, linewidth=1pt,linecolor=black](5.25,1.3){.23}
			\pscircle[fillstyle=solid,fillcolor=white, linewidth=1pt,linecolor=black](5.75,1.3){.23}

\psline(6.9,0)(8.6,0)
\psline(6.9,0)(6.9,2)
\psline(8.6,0)(8.6,2)
			\pscircle[fillstyle=solid,fillcolor=black, linewidth=1pt,linecolor=black](7.25,.3){.23}
			\pscircle[fillstyle=solid,fillcolor=black, linewidth=1pt,linecolor=black](7.75,.3){.23}
			\pscircle[fillstyle=solid,fillcolor=black, linewidth=1pt,linecolor=black](8.25,.3){.23}
			\pscircle[fillstyle=solid,fillcolor=black, linewidth=1pt,linecolor=black](7.25,.8){.23}
			\pscircle[fillstyle=solid,fillcolor=black, linewidth=1pt,linecolor=black](7.75,.8){.23}
			\pscircle[fillstyle=solid,fillcolor=black, linewidth=1pt,linecolor=black](8.25,.8){.23}
			\pscircle[fillstyle=solid,fillcolor=white, linewidth=1pt,linecolor=black](7.25,1.3){.23}
			\pscircle[fillstyle=solid,fillcolor=white, linewidth=1pt,linecolor=black](7.75,1.3){.23}
\psline(9.9,0)(11.6,0)
\psline(9.9,0)(9.9,2)
\psline(11.6,0)(11.6,2)
			\pscircle[fillstyle=solid,fillcolor=white, linewidth=1pt,linecolor=black](10.25,.3){.23}
			\pscircle[fillstyle=solid,fillcolor=white, linewidth=1pt,linecolor=black](10.75,.3){.23}
			\pscircle[fillstyle=solid,fillcolor=white, linewidth=1pt,linecolor=black](11.25,.3){.23}
			\pscircle[fillstyle=solid,fillcolor=white, linewidth=1pt,linecolor=black](10.25,.8){.23}

\psline(11.9,0)(13.6,0)
\psline(11.9,0)(11.9,2)
\psline(13.6,0)(13.6,2)
			\pscircle[fillstyle=solid,fillcolor=black, linewidth=1pt,linecolor=black](12.25,.3){.23}
			\pscircle[fillstyle=solid,fillcolor=black, linewidth=1pt,linecolor=black](12.75,.3){.23}
			\pscircle[fillstyle=solid,fillcolor=black, linewidth=1pt,linecolor=black](13.25,.3){.23}
			\pscircle[fillstyle=solid,fillcolor=black, linewidth=1pt,linecolor=black](12.25,.8){.23}

	\endpspicture

\noindent
{\footnotesize Figure 5:
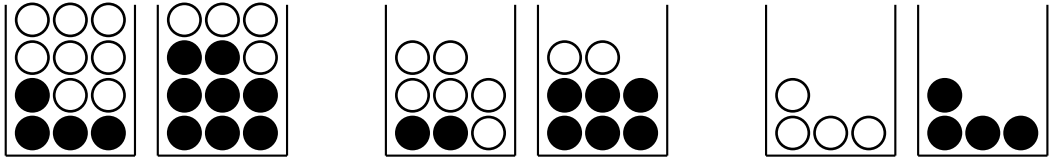
Several examples of two urns in configuration $C_1(4)$.}
\setcounter{figure}{5}

\vspace{0.3cm}
\noindent
	\end{center}
The configuration in which both urns are empty is denoted by $C_1(0)$.

\textbf{Configuration 2.} $B_1+B_2 \neq W_1+W_2$ and $(B_2-W_2)(B_1-W_1)\leq 0$. This means that there is a global majority color in the two urns, but one of the two urns has a majority of the other color. Suppose without loss of generality that $B_1+B_2> W_1+W_2$, $B_1>W_1$ and $B_2\leq W_2$. Note that starting from this configuration, all the balls drawn from the urn 1 will be black almost surely. Therefore, the law of the next ball to be drawn starting from this configuration only depends on the difference $W_2-B_2$. We denote this configuration by $C_2(W_2-B_2)$.
\begin{center}
	\pspicture(0,-.6)(13.6,2.5)

\psline(-0.1,0)(1.6,0)
\psline(-0.1,0)(-0.1,2)
\psline(1.6,0)(1.6,2)
			\pscircle[fillstyle=solid,fillcolor=black, linewidth=1pt,linecolor=black](.25,.3){.23}
			\pscircle[fillstyle=solid,fillcolor=black, linewidth=1pt,linecolor=black](.75,.3){.23}
			\pscircle[fillstyle=solid,fillcolor=black, linewidth=1pt,linecolor=black](1.25,.3){.23}
			\pscircle[fillstyle=solid,fillcolor=black, linewidth=1pt,linecolor=black](.25,.8){.23}
			\pscircle[fillstyle=solid,fillcolor=black, linewidth=1pt,linecolor=black](.75,.8){.23}
			\pscircle[fillstyle=solid,fillcolor=black, linewidth=1pt,linecolor=black](1.25,.8){.23}
			\pscircle[fillstyle=solid,fillcolor=black, linewidth=1pt,linecolor=black](.25,1.3){.23}
			\pscircle[fillstyle=solid,fillcolor=black, linewidth=1pt,linecolor=black](.75,1.3){.23}
			\pscircle[fillstyle=solid,fillcolor=black, linewidth=1pt,linecolor=black](1.25,1.3){.23}
			\pscircle[fillstyle=solid,fillcolor=black, linewidth=1pt,linecolor=black](.25,1.8){.23}
			\pscircle[fillstyle=solid,fillcolor=white, linewidth=1pt,linecolor=black](.75,1.8){.23}
			\pscircle[fillstyle=solid,fillcolor=white, linewidth=1pt,linecolor=black](1.25,1.8){.23}

\psline(1.9,0)(3.6,0)
\psline(1.9,0)(1.9,2)
\psline(3.6,0)(3.6,2)
			\pscircle[fillstyle=solid,fillcolor=black, linewidth=1pt,linecolor=black](2.25,.3){.23}
			\pscircle[fillstyle=solid,fillcolor=black, linewidth=1pt,linecolor=black](2.75,.3){.23}
			\pscircle[fillstyle=solid,fillcolor=black, linewidth=1pt,linecolor=black](3.25,.3){.23}
			\pscircle[fillstyle=solid,fillcolor=black, linewidth=1pt,linecolor=black](2.25,.8){.23}
			\pscircle[fillstyle=solid,fillcolor=black, linewidth=1pt,linecolor=black](2.75,.8){.23}
			\pscircle[fillstyle=solid,fillcolor=white, linewidth=1pt,linecolor=black](3.25,.8){.23}
			\pscircle[fillstyle=solid,fillcolor=white, linewidth=1pt,linecolor=black](2.25,1.3){.23}
			\pscircle[fillstyle=solid,fillcolor=white, linewidth=1pt,linecolor=black](2.75,1.3){.23}
			\pscircle[fillstyle=solid,fillcolor=white, linewidth=1pt,linecolor=black](3.25,1.3){.23}
			\pscircle[fillstyle=solid,fillcolor=white, linewidth=1pt,linecolor=black](2.25,1.8){.23}
			\pscircle[fillstyle=solid,fillcolor=white, linewidth=1pt,linecolor=black](2.75,1.8){.23}
			\pscircle[fillstyle=solid,fillcolor=white, linewidth=1pt,linecolor=black](3.25,1.8){.23}

\psline(4.9,0)(6.6,0)
\psline(4.9,0)(4.9,2)
\psline(6.6,0)(6.6,2)
			\pscircle[fillstyle=solid,fillcolor=black, linewidth=1pt,linecolor=black](5.25,.3){.23}
			\pscircle[fillstyle=solid,fillcolor=black, linewidth=1pt,linecolor=black](5.75,.3){.23}
			\pscircle[fillstyle=solid,fillcolor=black, linewidth=1pt,linecolor=black](6.25,.3){.23}
			\pscircle[fillstyle=solid,fillcolor=black, linewidth=1pt,linecolor=black](5.25,.8){.23}
			\pscircle[fillstyle=solid,fillcolor=black, linewidth=1pt,linecolor=black](5.75,.8){.23}
			\pscircle[fillstyle=solid,fillcolor=black, linewidth=1pt,linecolor=black](6.25,.8){.23}
			\pscircle[fillstyle=solid,fillcolor=black, linewidth=1pt,linecolor=black](5.25,1.3){.23}
			\pscircle[fillstyle=solid,fillcolor=black, linewidth=1pt,linecolor=black](5.75,1.3){.23}

\psline(6.9,0)(8.6,0)
\psline(6.9,0)(6.9,2)
\psline(8.6,0)(8.6,2)
			\pscircle[fillstyle=solid,fillcolor=white, linewidth=1pt,linecolor=black](7.25,.3){.23}
			\pscircle[fillstyle=solid,fillcolor=white, linewidth=1pt,linecolor=black](7.75,.3){.23}
			
\psline(9.9,0)(11.6,0)
\psline(9.9,0)(9.9,2)
\psline(11.6,0)(11.6,2)
			\pscircle[fillstyle=solid,fillcolor=black, linewidth=1pt,linecolor=black](10.25,.3){.23}
			\pscircle[fillstyle=solid,fillcolor=black, linewidth=1pt,linecolor=black](10.75,.3){.23}
			\pscircle[fillstyle=solid,fillcolor=black, linewidth=1pt,linecolor=black](11.25,.3){.23}
			\pscircle[fillstyle=solid,fillcolor=black, linewidth=1pt,linecolor=black](10.25,.8){.23}
			\pscircle[fillstyle=solid,fillcolor=black, linewidth=1pt,linecolor=black](10.75,.8){.23}
			\pscircle[fillstyle=solid,fillcolor=black, linewidth=1pt,linecolor=black](11.25,.8){.23}
			\pscircle[fillstyle=solid,fillcolor=black, linewidth=1pt,linecolor=black](10.25,1.3){.23}
			\pscircle[fillstyle=solid,fillcolor=black, linewidth=1pt,linecolor=black](10.75,1.3){.23}
			\pscircle[fillstyle=solid,fillcolor=black, linewidth=1pt,linecolor=black](11.25,1.3){.23}
			\pscircle[fillstyle=solid,fillcolor=black, linewidth=1pt,linecolor=black](10.25,1.8){.23}
			\pscircle[fillstyle=solid,fillcolor=black, linewidth=1pt,linecolor=black](10.75,1.8){.23}
			\pscircle[fillstyle=solid,fillcolor=black, linewidth=1pt,linecolor=black](11.25,1.8){.23}

\psline(11.9,0)(13.6,0)
\psline(11.9,0)(11.9,2)
\psline(13.6,0)(13.6,2)
			\pscircle[fillstyle=solid,fillcolor=white, linewidth=1pt,linecolor=black](12.25,.3){.23}
			\pscircle[fillstyle=solid,fillcolor=white, linewidth=1pt,linecolor=black](12.75,.3){.23}
			
	\endpspicture

\noindent
{\footnotesize Figure 6:
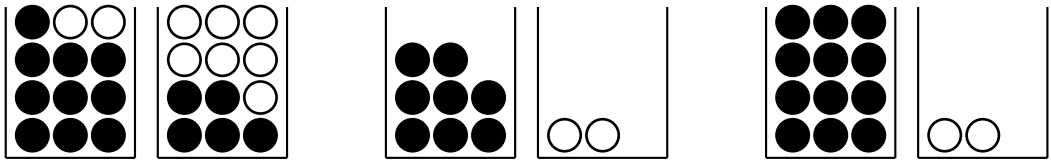
Several examples of two urns in configuration $C_2(2)$.}
\setcounter{figure}{6}

\vspace{0.3cm}
\noindent
	\end{center}
	
\textbf{Configuration 3.} $B_1+B_2\neq W_1+W_2$, and $(B_2-W_2)(B_1-W_1)>0$. Here the global majority color is the same as the majority color in each of the two urns. Starting from this configuration, all the balls drawn will have the majority color, almost surely. We denote this configuration by $C_3$.

Since the two urns are initially empty, they stay for some time in Configuration 1 and then jump to Configuration 2 (or possibly directly to Configuration 3). Once in Configuration 2, they can either stay in Configuration 2 forever, or jump to Configuration 3, and stay there forever after. 
A typical evolution of this process is schematically depicted by the following figure:

\begin{center}
	\pspicture(0,-2)(11,2)

\rput(1.5,-.5){$\underbrace{~~~~~~~~~~~~~~~~~~~~}$}
\rput(1.5,-1){Configuration 1}
\rput(5.1,-.5){$\underbrace{~~~~~~~~~~~~~~~~~~~~~~~~~~~~}$}
\rput(5.1,-1){Configuration 2}
\rput(8.9,-.5){$\underbrace{~~~~~~~~~~~~~~~~~~~~~~}$}
\rput(8.9,-1){Configuration 3}
\psline[border=1.5pt](0,0)(10,0)
\psline[border=1.5pt](0,0)(0,.6)
\psline[border=1.5pt](0,.6)(10,.6)

\psline[border=1.5pt](0,1)(10,1)
\psline[border=1.5pt](0,1)(0,1.6)
\psline[border=1.5pt](0,1.6)(10,1.6)

			\pscircle[fillstyle=solid,fillcolor=black, linewidth=1pt,linecolor=black](.3,.3){.2}
			\pscircle[fillstyle=solid,fillcolor=white, linewidth=1pt,linecolor=black](.3,1.3){.2}

			\pscircle[fillstyle=solid,fillcolor=black, linewidth=1pt,linecolor=black](.9,.3){.2}
			\pscircle[fillstyle=solid,fillcolor=white, linewidth=1pt,linecolor=black](.9,1.3){.2}

			\pscircle[fillstyle=solid,fillcolor=white, linewidth=1pt,linecolor=black](1.5,.3){.2}
			\pscircle[fillstyle=solid,fillcolor=black, linewidth=1pt,linecolor=black](1.5,1.3){.2}

			\pscircle[fillstyle=solid,fillcolor=black, linewidth=1pt,linecolor=black](2.1,.3){.2}
			\pscircle[fillstyle=solid,fillcolor=white, linewidth=1pt,linecolor=black](2.1,1.3){.2}

			\pscircle[fillstyle=solid,fillcolor=white, linewidth=1pt,linecolor=black](2.7,.3){.2}
			\pscircle[fillstyle=solid,fillcolor=black, linewidth=1pt,linecolor=black](2.7,1.3){.2}

			\pscircle[fillstyle=solid,fillcolor=white, linewidth=1pt,linecolor=black](3.3,.3){.2}
			\pscircle[fillstyle=solid,fillcolor=white, linewidth=1pt,linecolor=black](3.3,1.3){.2}

			\pscircle[fillstyle=solid,fillcolor=black, linewidth=1pt,linecolor=black](3.9,.3){.2}
			\pscircle[fillstyle=solid,fillcolor=white, linewidth=1pt,linecolor=black](3.9,1.3){.2}

			\pscircle[fillstyle=solid,fillcolor=white, linewidth=1pt,linecolor=black](4.5,.3){.2}
			\pscircle[fillstyle=solid,fillcolor=white, linewidth=1pt,linecolor=black](4.5,1.3){.2}

			\pscircle[fillstyle=solid,fillcolor=black, linewidth=1pt,linecolor=black](5.1,.3){.2}
			\pscircle[fillstyle=solid,fillcolor=white, linewidth=1pt,linecolor=black](5.1,1.3){.2}

			\pscircle[fillstyle=solid,fillcolor=black, linewidth=1pt,linecolor=black](5.7,.3){.2}
			\pscircle[fillstyle=solid,fillcolor=white, linewidth=1pt,linecolor=black](5.7,1.3){.2}

			\pscircle[fillstyle=solid,fillcolor=white, linewidth=1pt,linecolor=black](6.3,.3){.2}
			\pscircle[fillstyle=solid,fillcolor=white, linewidth=1pt,linecolor=black](6.3,1.3){.2}

			\pscircle[fillstyle=solid,fillcolor=white, linewidth=1pt,linecolor=black](6.9,.3){.2}
			\pscircle[fillstyle=solid,fillcolor=white, linewidth=1pt,linecolor=black](6.9,1.3){.2}

			\pscircle[fillstyle=solid,fillcolor=white, linewidth=1pt,linecolor=black](7.5,.3){.2}
			\pscircle[fillstyle=solid,fillcolor=white, linewidth=1pt,linecolor=black](7.5,1.3){.2}

			\pscircle[fillstyle=solid,fillcolor=white, linewidth=1pt,linecolor=black](8.1,.3){.2}
			\pscircle[fillstyle=solid,fillcolor=white, linewidth=1pt,linecolor=black](8.1,1.3){.2}

			\pscircle[fillstyle=solid,fillcolor=white, linewidth=1pt,linecolor=black](8.7,.3){.2}
			\pscircle[fillstyle=solid,fillcolor=white, linewidth=1pt,linecolor=black](8.7,1.3){.2}
			
	\endpspicture

\noindent
{\footnotesize Figure 7:
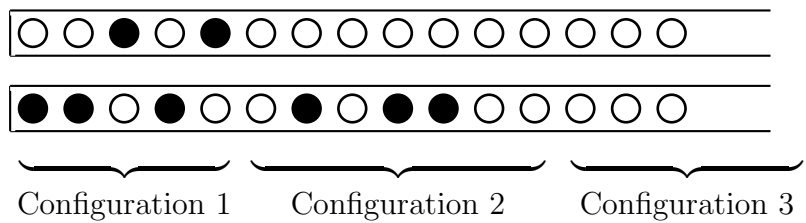
The three phases of the evolution.}
\setcounter{figure}{7}

\vspace{0.3cm}
\noindent
	\end{center}

\subsection{Proof of Theorem \ref{theorem}}
For $\ell\geq 0$ we denote by
$$q_\ell\equiv q_\ell(p):=\mathbb{P}\left[\text{The two urns fixate on the same color starting from configuration }C_1(\ell)\right],$$
and
$$r_\ell\equiv r_\ell(p):=\mathbb{P}\left[\text{The two urns fixate on the same color starting from configuration }C_2(\ell)\right].$$

As already noted, starting from configuration $C_3$, the probability that the two urns fixate on the same color equals 1. 
We proceed by observing that, due to the Markov-like property, some simple relations connect these different probabilities.

Starting from $C_1(0)$, each urn independently draws a white or black ball with probability 1/2. Therefore with probability 1/2, the two urns draw the same color and immediately (at time $1$) enter the configuration $C_3$, and with probability 1/2 the two urns choose different colors and jump to configuration $C_1(1)$ at time $1$. This implies
\begin{equation}
\label{eqn_p0}q_0=\frac{1}{2}+\frac{1}{2}q_1.
\end{equation}
For $\ell\geq 1$, starting from $C_1(\ell)$, the urns can jump to $C_1(\ell-1)$, $C_1(\ell+1)$ or $C_2(\ell-1)$, and this decision is made independently form the past. More precisely:
\begin{itemize}
\item They jump to configuration $C_1(\ell-1)$ if and only if each urn draws a ball of the color which is not its majority color. In order that this happens, for each of the urns the drawing must be done in the two urns combined (with probability $p$) and the minority color of the urns combined must be chosen (with probability $1/2$). So the probability of this move equals $(p/2)^2$.
\item They jump to configuration $C_1(\ell+1)$ if and only if each urn draws a ball of its majority color. This happens with probability $((1-p)+p/2)^2=(1-p/2)^2$.
\item They jump to configuration $C_2(\ell-1)$ if the two urns draw a ball of the same color. This happens with the remaining probability $p(1-p/2)$.
\end{itemize}
This implies the following sequence of relations:
\begin{equation}\label{eqn_p1}
q_\ell=\left(\frac{p}{2}\right)^2 q_{\ell-1}+\left(1-\frac{p}{2}\right)^2 q_{\ell+1}+p(1-\frac{p}{2})r_{\ell-1}, \ \ell\geq 1.
\end{equation}

Starting from configuration $C_2(0)$, the urns either jump to configuration $C_3$ with probability $(1+p)/2$, or to configuration $C_2(1)$ with probability $(1-p)/2$. 
We conclude that
\begin{equation}\label{eqn_r1}
r_0= \frac{1+p}{2}+\frac{1-p}{2}r_1.
\end{equation}

Similarly, for $\ell\geq 1$, starting from configuration $C_2(\ell)$, the urns either jump to configuration $C_2(\ell-1)$ with probability $p$, or to configuration $C_2(\ell-1)$ with probability $1-p$, again independently of the past. This translates into
\begin{equation}\label{eqn_r2}
r_\ell = p r_{\ell-1} + (1-p)r_{\ell+1}.
\end{equation}
In fact, the above reasoning shows that if $Y_n$ is an integer-valued random variable such that at time $n$, the urns are in configuration $C_2(Y_n)$, then $(Y_{k \wedge \tau_0}, \,k\geq n)$ is the simple random walk with drift $1-2p$ killed at 
$\tau_0:=\inf\{j\geq n:\, Y_j=0\}$.
It is well known (recall that $p<1/2$) that 
$\P(\tau_0<\iy\,|\, \FF_n)=p/(1-p)$ on $\{Y_n=1\}$, yielding
$r_1=\frac{p}{1-p}r_0$, and
together with (\ref{eqn_r1}) this  gives
$r_0=\frac{1+p}{2-p}.$
Now one can deduce by induction from (\ref{eqn_r2}) that 
$$r_\ell=\frac{1+p}{2-p}\left(\frac{p}{1-p}\right)^\ell,\ \ell\geq 0.$$
We then introduce these values into (\ref{eqn_p1}) to obtain
\begin{equation}\label{eqn_p2}
q_\ell=\left(\frac{p}{2}\right)^2 q_{\ell-1}+\left(1-\frac{p}{2}\right)^2 q_{\ell+1}+\frac{p(1+p)}{2}\left(\frac{p}{1-p}\right)^{\ell-1}, \ \ell\geq 1.
\end{equation}
Define a function
$$f_p(x):=\sum_{\ell=0}^{\iy}q_\ell(p)\frac{x^\ell}{\ell !}.$$
Differentiating both sides and applying (\ref{eqn_p2}) we get
\begin{eqnarray*}
f_p^{\prime}(x)&=&\sum_{\ell=0}^{\iy}q_\ell(p)\frac{x^{\ell-1}}{(\ell-1) !}\\
&=&\sum_{\ell=1}^{\iy}\left(\left(\frac{p}{2}\right)^2 q_{\ell-1}+\left(1-\frac{p}{2}\right)^2 q_{\ell+1}+\frac{p(1+p)}{2}\left(\frac{p}{1-p}\right)^{\ell-1}\right)\frac{x^{\ell-1}}{(\ell-1) !}\\
&=&\left(\frac{p}{2}\right)^2\sum_{\ell=0}^{\iy}q_\ell(p)\frac{x^{\ell}}{\ell !}+\left(1-\frac{p}{2}\right)^2\sum_{\ell=0}^{\iy}q_{\ell+2}(p)\frac{x^{\ell}}{\ell !}+\frac{p(1+p)}{2}\sum_{\ell=0}^{\iy}\frac{\left(\frac{px}{1-p}\right)^{\ell}}{\ell !}\\
&=&\left(\frac{p}{2}\right)^2f_p(x)+\left(1-\frac{p}{2}\right)^2f_p^{\prime\prime}(x)+\frac{p(1+p)}{2}\exp\left(\frac{px}{1-p}\right).\\
\end{eqnarray*}
So $f_p$ is a solution of the ODE
\begin{equation}\label{ODE}
\left(1-\frac{p}{2}\right)^2y^{\prime\prime}-y^{\prime}+\left(\frac{p}{2}\right)^2y+\frac{p(1+p)}{2}\exp\left(\frac{px}{1-p}\right)=0.
\end{equation}
The characteristic polynomial of the homogeneous ODE is
$$\left(1-\frac{p}{2}\right)^2\lambda^2-\lambda+\left(\frac{p}{2}\right)^2=0,$$
and its two solutions are
\begin{equation}\label{Elambda}
\lambda_{\pm}(p)=\frac{1\pm\sqrt{1-4\left(\frac{p}{2}\right)^2\left(1-\frac{p}{2}\right)^2}}{2\left(1-p/2\right)^2}=\frac{1\pm\sqrt{1-p^2\left(1-\frac{p}{2}\right)^2}}{2\left(1-p/2\right)^2}.
\end{equation}
So the general solution of the homogeneous equation is of the form
$$f_p^1(x)=A(p)\exp{(\lambda_{-}(p)x)}+B(p)\exp{(\lambda_{+}(p)x)}.$$
Let us now search for a particular solution to (\ref{ODE}) of the form $f_p^2(x)=C(p)\exp(xp/(1-p))$. Plugging it into (\ref{ODE}), we obtain
\begin{equation}\label{EC}
C(p)=-\frac{2 (1-p)^2 (1+p)}{2 p^3-6p^2+9p-4}\ .
\end{equation}
Hence the general solution to (\ref{ODE}) is of the form
$$f_p(x)=f_p^1(x)+f_p^2(x)=A(p)\exp{(\lambda_{-}(p)x)}+B(p)\exp{(\lambda_{+}(p)x)}+C(p)\exp\left(\frac{p}{1-p}x\right).$$
In order to find $A(p)$ and $B(p)$, note that since $q_\ell\leq 1$ for all $\ell\geq 0$, we have
$$f_p(x)=\sum_{\ell=0}^{\iy}q_\ell(p)\frac{x^\ell}{\ell !}\leq \sum_{\ell=0}^{\iy}\frac{x^\ell}{\ell !}=\exp(x).$$
So $f_p(x)=o(\exp{(\lambda_{+}(p)x)})$ since $\lambda_+(p)>1$ as can be easily checked.
Therefore we deduce that $B(p)=0$.
We can now find $A(p)$ from the following linear system:
$$\left\{\begin{array}{l}
q_0=f_p(0)=A(p)+C(p)\\
q_1=f_p^{\prime}(0)=A(p)\lambda_{-}(p)+C(p)\frac{p}{1-p}
\end{array}\right. .$$
Recalling (\ref{eqn_p0}) we then obtain
$$A(p)+C(p)=\frac{1}{2}+\frac{1}{2}\left(A(p)\lambda_{-}(p)+C(p)\frac{p}{1-p}\right),$$
and deduce that
$$A(p)=\frac{1-p+C(p)(3p-2)}{(1-p)(2-\lambda_{-}(p))}.$$
Therefore
$$q_0(p)=A(p)+C(p)=\frac{1-p+C(p)(3p-2)}{(1-p)(2-\lambda_{-}(p))}+C(p),$$
where $C(p)$ and $\lambda_{-}(p)$ are as in (\ref{EC}) and (\ref{Elambda}), respectively, concluding the proof of Theorem \ref{theorem}. 

One can check in particular that $C(0)=1/2$, $C(1/2)=1$, $A(0)=A(1/2)=0$, and so
$$q_0(0)=1/2\text{~~~and~~~}q_0(1/2)=1.$$
Indeed, when $p=0$ the two urns are independent and the probability that they fixate on the same color is $1/2$, while $p=1/2$ is the critical phase of the IUM, for which we already knew that the two urns fixate on the same color, almost surely (see Theorem 3.5 in \cite{launay1}).

\subsection{Proof of Theorem \ref{T:approx}}
\label{S:approx}
In most of this argument we do not fix some $\rho\in (1,\iy]$ but rather consider the whole family of laws $\P_\rho$, $\rho\in (1,\iy]$.
Denote by $S_k$ the state of the two (generalized) interacting urns at time $k$. 
The reader should think of $S_k$ as a ``struct'' that contains the count of black and white balls for each of the urns at time $k$, as well as the information on the drawing pools used at time $k-1$ (for the draws made to update the state at time $k$ from the state at time $k-1$). 
Denote by $\FF=(\FF_k)_{k \geq 0}$ the filtration generated by the process $S$, that is, 
$\FF_k=\sigma\{S_n:n =0,\ldots,k\}$.
Define furthermore $\sigma_1:=0$,
$$
\sigma_2:=\inf\{k \geq 0: S_k \in \cup_{i=1}^\iy C_2(i)\},\
\sigma_3:=\inf\{k \geq 0: S_k \in C_3\}.
$$
Clearly, all the $\sigma_i$ are stopping times with respect to $\FF$.
As depicted in Figure 7, we have $\P_\iy(\sigma_1 \leq \sigma_2 \wedge \sigma_3 \leq  \sigma_3)=1$.
As argued in the course of the proof of Theorem \ref{theorem}, we also have that 
if $S_k \in C_2(\ell)$ for some $k$ and $\ell\geq 1$, then
$\P_\iy(S_{k+1} \in \{C_2(\ell-1),C_2(\ell+1)\})=1$.  
Similarly, if $S_k \in C_2(0)$ for some $k$ then 
$ \P_\iy(S_{k+1} \in \{C_2(1),C_3\})=1$, and if $S_k \in C_3$ then
$\P_\iy(S_{k+1} \in C_3)=1$. Therefore
\[
F= \{\sigma_3< \iy\},\ \P_\iy-\mbox{almost surely}.
\]
Our main problem in proving the current result is related to the fact that most of the above statements fail
to hold under $\P_\rho$ if $\rho<\iy$.

Let $\bar{A}$ be following event: there exist some $m \in \N$ and an urn $u$ such that the color of the ball drawn for $u$ at time $m$ is the minority color in the drawing pool  prescribed by the corresponding Bernoulli($p$) flip (if $1$ this pool is the two urns combined, otherwise this pool is the urn $u$ itself). Such a draw is asymptotically impossible, or equivalently $\P_\iy(\bar{A})=0$, and henceforth we will refer to it as ``AI-draw''. Let $A$ be the complement of $\bar{A}$.
Our strategy is to show the following two points.
\begin{lemme}
\label{L:two pts}
(i) $\lim_{\rho\to\iy}\P_\rho[F \cap A]=\P_\iy[F \cap A]=\P_\iy[F]$,\\
(ii) $\P_\rho[F\cap \bar{A}]=O(1/\rho)$.
\end{lemme}
Let $\tau:= \inf\{k \geq 0: S_k$ is a result of an AI-draw$\}$, with the usual convention that $\tau=\iy$ if such a draw is never made.
Then is it clear that $\tau$ is a stopping time with respect to $\FF$ (for each $\rho\in (1,\iy)$).

Our interest in $\tau$ comes from the observation that 
the state of the two urns changes according to the same general pattern (depicted in Figure 7) under $\P_\rho$ and under $\P_\infty$ until the time of the first AI-draw, where the transition probabilities depend on $\rho$.
More precisely, we have $\sigma_1=\sigma_1\wedge \tau \leq \tau \wedge \sigma_2 
\wedge \sigma_3 \leq \tau \wedge \sigma_3$, and also if $S_k \in C_2(\ell)$ for some $k<\tau$ and $\ell\geq 1$, then either 
$\{\tau=k+1\}$ or $\{\tau > k+1, S_{k+1} \in \{C_2(\ell-1),C_2(\ell+1)\}\}$ happen
$\P_\rho$-almost surely.  
Furthermore, if $S_k \in C_2(0)$ for some $k<\tau$, then 
$\P_\rho(\tau=k+1) + \P_\rho(\tau > k+1, S_{k+1} \in \{C_2(1),C_3\})=1$
and if $S_k \in C_3$ for some $k<\tau$, then 
$\P_\rho(\tau=k+1) + \P_\rho(\tau > k+1, S_{k+1} \in C_3)=1$.
Therefore, $F\cap A=\{\sigma_3<\iy, \tau=\iy\}$ and more importantly $F\cap \bar{A}=\{\sigma_3<\iy, \tau<\iy\}$, $\P_\rho$-almost surely, for any $\rho\in(1,\iy]$.

In order to show part (i) of Lemma \ref{L:two pts}, we note that the probability of any path (in the countable set of infinite paths) contributing to $F\cap A$
is a product of finitely many terms of the form $c\times\frac{\rho^a}{\rho^a+\rho^b}$ with 
$c\in \{1,p,1-p\}$ and
$a\geq b$ (otherwise, there would exist an AI-draw along the path)
and infinitely many terms of the form $p\times\frac{\rho^a}{\rho^a+\rho^b} +
(1-p)\times\frac{\rho^c}{\rho^c+\rho^d}$, where again $a/b \wedge c/d\geq 1$.
Therefore for each infinite path $\pi\in F\cap A$, 
$$\rho \mapsto \P_{\rho}(\pi)$$
increases to $\P_{\iy}(\pi)$ and so by the monotone convergence theorem, we have $\lim_{\rho\to\iy}\P_\rho(F\cap A)=\P_\iy(F\cap A)=\P_\iy(F)$, where the last equality is due to 
$\P_\iy(\bar{A})=0$. 

Arguing part (ii) of Lemma \ref{L:two pts} turns out to be more tricky.
First note that $\P_\rho(F\cap \bar{A})=$
\[
\P_\rho(\sigma_3<\iy, \tau<\iy)\leq \P_\rho(\tau \leq \sigma_2 \wedge \sigma_3,\sigma_3<\iy) + \P_\rho(\sigma_2< \tau \leq \sigma_3<\iy) + \P_\rho(\sigma_3< \tau <\iy).
\]
The final term above is the simplest one to study so we consider it first.
Indeed, at time $\sigma_3$ 
the count of the global majority color for all the pools (each of the urns, and therefore the urns combined) is at least by one higher than that for the minority color.
Moreover,
as long as the two urns remain in state $C_3$ (which is certainly true until time $\tau$), the count of the majority color for all the pools (each of the urns, and therefore the urns combined) increases by at least one in each step.
Therefore, by nested conditioning with respect to $\FF_{\sigma_3+i}$, we have
\begin{eqnarray*}
\P_\rho(\sigma_3< \tau <\iy)&=& 
\sum_{i=0}^\iy \P_\rho(\sigma_3<\iy,\sigma_3+i<\tau,\sigma_3+i+1= \tau)\\
&\leq& \sum_{i=0}^\iy \P_\rho(\sigma_3<\iy,\sigma_3+i<\tau) \cdot \frac{1}{1+\rho^{i+1}}\\
&\leq& \sum_{i=0}^\iy \frac{1}{1+\rho^{i+1}}= O\left(\frac{1}{\rho} \right).
\end{eqnarray*}

Note that $\sigma_2 \wedge \sigma_3$ is a finite (even stochastically bounded by a 
geometric, see the proof of Lemma \ref{L:unibd on tau} below) random variable under $\P_\rho$, for each $\rho$.
We split the event $\{\tau\leq \sigma_2 \wedge \sigma_3\}$ according to the state of the system at time $\tau-1$ and the value of $\tau$. Note that AI-draw is impossible from state $C_1(0)$, for any $\rho\in \R_+$. Also note that $\sigma_2<\sigma_3$ (in words, the two urns enter $\cup_{\ell \geq 0}C_2(\ell)$ at the exit time from $\cup_i C_1(i)$) unless 
$S_{\sigma_2 \wedge \sigma_3 -1}\in C_1(0)$ in which case the two urns jump directly to $C_3$. 
Thus
\begin{eqnarray*}
\P_\rho(\tau\leq \sigma_2 \wedge \sigma_3,\sigma_3<\iy)&=& 
\sum_{i=1}^\iy \P_\rho(S_{\tau-1}\in C_1(i), \tau\leq \sigma_2, \sigma_3<\iy)\\
&\leq& \sum_{i=1}^\iy \sum_{k=0}^\iy \P_\rho(\tau>k-1, S_{k-1}\in C_1(i),
\tau=k,\sigma_3<\iy)\\
&\leq& \sum_{i=1}^\iy \sum_{k=0}^\iy
\frac{1}{1+\rho^i} \P_\rho(\tau>k-1, S_{k-1}\in C_1(i),\,\sigma_3<\iy)\\
&\leq& \sum_{i=1}^\iy \frac{1}{1+\rho^i} \E(\#\{\mbox{steps in }C_1(i) 
\mbox{ before }\tau\};\,\sigma_3<\iy)\\
&\leq& O\left(\frac{1}{\rho} \right) \E(\#\{\mbox{steps in }\cup_i C_1(i)
 \mbox{ before }\tau\};\,\sigma_3<\iy),
\end{eqnarray*}
where the inequality in the third line is again obtained by nested conditioning with respect to $\FF_{k-1}$ and the fact that on $\{S_{k-1}\in C_1(i)\}$ the difference in the counts of majority and minority colors is equal to $i$.

In an analogous fashion (we leave the details to the reader) one can obtain
$$
\P_\rho(\sigma_2< \tau \leq \sigma_3<\iy) \leq 
O\left(\frac{1}{\rho} \right) \E(\#\{\mbox{steps in }\cup_i C_2(i)
 \mbox{ before }\tau\};\,\sigma_3<\iy).
$$

\begin{lemme}
\label{L:unibd on tau}
Fix $p\in (0,1/2)$.
There exist $\rho_0<\iy$ such that for each $\rho > \rho_0$ we have
$$
\E_\rho(\tau \wedge \sigma_3 \cdot \indica{\sigma_3<\iy}) < C,
$$
where $C$ is a finite constant which depends only on $p$.
\end{lemme}

Since, due to the above calculations,
$$
\P_\rho(F\cap \bar{A})=O\left(\frac{1}{\rho} \right) (1+ \E(\#\{\mbox{steps in }\cup_{j=1,2}\cup_i C_j(i) \mbox{ before }\tau\};\,\sigma_3<\iy),
$$
and the expectation on the right-hand side above is precisely 
$\E_\rho(\tau \wedge \sigma_3 \cdot \indica{\sigma_3<\iy})$, we obtained the required claim, and this completes the proof of Lemma \ref{L:two pts}, and in turn the proof of Theorem \ref{T:approx}.

\smallskip
\noindent
{\em Proof of Lemma \ref{L:unibd on tau}.}
Note that
$$
\E_\rho(\tau \wedge \sigma_3 \cdot \indica{\sigma_3<\iy}) \leq 
\E_\rho(\tau \wedge \sigma_2 \wedge \sigma_3) + \E_\rho((\tau \wedge \sigma_3 - \tau \wedge \sigma_2 \wedge \sigma_3) \,\indica{\sigma_3<\iy}).
$$
Again, bounding the first expectation on the right-hand side above is simpler. For this, note that on $\{k < \sigma_2 \wedge \sigma_3\}$ the two urns have a probability uniformly bounded from below by $2(p/2)^2=p^2/2>0$ for $\sigma_2$ to happen at step $k$, since for this it is sufficient that the draws for both urns are made from the two urns combined, and that the same (equally likely) color is chosen for both.
Therefore, as already mentioned, the random variable $\sigma_2 \wedge \sigma_3$ is stochastically bounded by a geometric random variable with success probability $p^2/2$. We conclude that
$\tau \wedge \sigma_2 \wedge \sigma_3$ has expectation bounded by $2/p^2$ under any $\P_\rho$, for $\rho\in (1,\iy]$.

To bound 
$ \E_\rho((\tau \wedge \sigma_3 - \tau \wedge \sigma_2 \wedge \sigma_3) \,\indica{\sigma_3<\iy})$, we use the fact that on $\{\sigma_2<\sigma_3\}$ 
(WLOG we can and will assume that this happens), under $\P_\rho$, the steps of the two urn model during the time interval $[\tau \wedge \sigma_2, \tau \wedge \sigma_3]$
go through the classes $C_2(i)$ according to the pattern illustrated by Figure 7.
More precisely, it is easily seen that 
\begin{equation}
\label{EenterC3}
\begin{array}{ll}
\P_{\rho}(\sigma_3=k+1\,|\, \FF_k) \indica{\tau > k,\, S_k \in C_2(0)}
\bck &
=\P_{\rho}(S_{k+1} \in C_3 \,|\, \FF_k) \indica{\tau > k,\, S_k \in C_2(0)}\\
& \geq \frac{1}{4}
\indica{\tau > k, S_k \in C_2(0)},
\end{array}
\end{equation}
uniformly in $k$, and moreover that for any $\ell \geq 1$ and any $k\in\N$.
on $\{\tau > k,\, S_k \in C_2(\ell)\}$
\begin{equation}
\label{EstepsC2}
\begin{array}{ll}
\P_{\rho}(S_{k+1} \in C_2(\ell+1),\tau>k+1 \,|\, \FF_k) =
(1-p) (1+O_{k,\ell+1}(\rho)),\\
\P_{\rho}(S_{k+1} \in C_2(\ell-1),\tau>k+1 \,|\, \FF_k) =
p (1+O_{k,\ell-1}(\rho)),
\end{array}
\end{equation}
where for each $k$ and all $\ell\in \N$, $O_{k,\ell+1}(\rho),O_{k,\ell-1}(\rho)$ 
are $\FF_k$-measurable random variables satisfying the following important 
property: there exists some $C<\iy$ such that 
\[
\P_\rho(O_{k,\ell-1}(\rho)\leq C/\rho)=\P_\rho(O_{k,\ell+1}(\rho)\leq C/\rho)=1, \ \forall k,\ell\in \N, \ \rho\in (1,\iy).
\]

Now fix some $\rho \in (1,\iy)$.
On $\{\sigma_2 < \tau \wedge \tau_3\}$, the two urns enter the set of configurations $\cup_\ell C_2(\ell)$
via $C_2(I)$, where $I$ is such that $S_{\sigma_2}\in C_2(I)$, $\P_\rho$-almost surely.
Due to the reasoning used in bounding $\E(\tau \wedge \sigma_2 \wedge \sigma_3)$, $I$ is stochastically bounded by a geometric random variable with success probability $p^2/2$.
Recall the process $Y_n$ defined in the paragraph which contains relation 
(\ref{eqn_r2}). 
Note that the same definition for $Y$ makes sense under $\P_\rho$ (at least) until time $\tau$, and recall the stopping time $\tau_0=\inf\{j\geq n:\, Y_{j\wedge \tau}=0\}$.

Due to (\ref{EenterC3}), there are at most geometric (with success probability $1/4$) many excursions of $Y_{\cdot \wedge \tau}$ above $0$ before $\sigma_3$.
The first excursion starts at $I$ and ends at $0$, while all the others start and end at $0$. 
Due to (\ref{EstepsC2}), all the excursions but (possibly) the first one 
are ``identically distributed up to an error of order
$(1+C/\rho)^{\text{excursion length}}$''.
By this we mean that for a fixed excursion path $\be$ away from $0$, the probability
that the $k$th excursion above takes value $\be$ equals 
$p(\be)x (1+g_k(\rho,\be))^{|\be|}$
where $p(\be)$ depends only on the steps in $\be$, $|\be|$ is the number of steps in 
$\be$, and $|g_k(\rho,\be)|\leq C/\rho$ for all $k,\rho,\be$. In the sequel, we will explain this even more precisely.
Due to Wald identity, we now know that for bounding 
$\E_\rho((\tau \wedge \sigma_3 - \tau \wedge \sigma_2 \wedge \sigma_3) \,\indica{\sigma_3<\iy})$, it suffices to show that
both 
\begin{equation}
\label{Etau from i}
\sum_i (1- p^2/2)^{i-1}\E_\rho(\tau_0;\tau_0<\iy|Y_0=i)<\iy
\end{equation}
and $\E(\tau_0;\tau_0<\iy|Y_0=0)<\iy$.
The excursion starting at $i$ and ending at $0$ can again be split into 
$i$ almost identically distributed excursions (due to (\ref{EstepsC2})),  
where the $j$th excursion starts at $i-j+1$ and ends at the time of 
the first visit to level $i-j$, $j=1,\ldots,i$.
Denote by $\tau_{j}$ the time of the first visit to level $j$.
Therefore, bounding $\E_\rho(\tau_0;\tau_0<\iy|Y_0=i)$ for $i\geq 1$ amounts to bounding
$\sum_{j=1}^i \E_\rho(\tau_{i-j}-\tau_{i-j+1};\tau_{i-j}<\iy|Y_0=i)$.

Now due to (\ref{EstepsC2}), we know that
\[
\E_\rho(\tau_{i-j}-\tau_{i-j+1};\tau_{i-j}<\iy|Y_0=i)
\leq \E(\tau_0'\cdot (1+C/\rho)^{\tau_0'};\tau_0'<\iy|Y_0'=1), 
\ 1\leq j \leq i,
\]
where $Y'$ is now the simple asymmetric random walk with $\P(\Delta Y' =1)=1-p=1-\P(\Delta Y' =-1)$ (as it is under $\P_\iy$) and $\tau_0'$ is its corresponding $\tau_0$.
With a finite upper bound $B$ for the expectation on the right-hand side in the last display, we would have
$\E_\rho(\tau_0;\tau_0<\iy|Y_0=i)\leq i \times B$ and this would make the sum in (\ref{Etau from i}) bounded by a multiple of $B$, that would depend only on $p$. 
Furthermore, since trivially $\E_\rho(\tau_0;\tau_0<\iy|Y_0=0)= 
1 + \E_\rho(\tau_0;\tau_0<\iy|Y_0=1)$ (the walk $Y$ is reflected at $0$), 
one sees that a uniform (in large $\rho$) upper bound on 
$$\E(\tau_0 \cdot (1+C/\rho)^{\tau_0'};\tau_0'<\iy|Y_0'=1)$$ 
is sufficient to conclude the argument.

The next step is to condition on $\tau_0'<\iy$. 
It is well-known (and easy to check) that the law of $Y'$ 
(simple asymmetric random walk with bias $1-2p>0$), conditioned on $\tau_0'<\iy$, 
becomes the law of the simple asymmetric random walk $Y''$ with bias $2p-1<0$, 
where $\P(\Delta Y'' =1)=p=1-\P(\Delta Y'' =-1)$.
Due to the well-known correspondence between the excursion of $Y''$ 
and the subcritical Galton-Watson branching (of geometric offspring distribution 
with success probability $(1-p)$ and mean $1/(1-p)-1<1$), 
one can easily see that there exists $\lambda>1$ such that
\[
\E(\lambda^{\tau_0''}) = \E(\lambda^{2 Z})<\iy,
\]
where $Z$ is the total population size of the above Galton-Watson process, 
started from a single individual.
Denote by $Z_m$ the total size od the first $m$ generations. Then,
by conditioning on the size of the first generation, one obtains that
\[
\E(\nu^{Z_m})= \nu \cdot (1-p) \cdot \frac{1}{1-p\cdot \E(\nu^{Z_{m-1}})}, \ m\geq 1, \, \nu>0.
\]
Consider 
$$
g= \nu \cdot (1-p) \cdot \frac{1}{1-p\cdot g},
$$
and note that 
if $\nu_0\equiv\nu_0(p)=(4p(1-p))^{-1}>1$ (recall that $p<1/2$) one can solve the above equation for all $\nu\in[0, \nu_0]$ to obtain a functional solution $g_p(\nu)$, continuously increasing in $\nu$.
Since clearly for $\nu\in (1,\nu_0]$ we have $\E(\nu^{Z_0})=\nu< g_p(\nu)$
and since $x<g_p(\nu)$ implies $(1-p) \cdot \frac{1}{1-p\cdot x}<g_p(\nu)$, one concludes that 
$\sup_m \E(\nu^{Z_m}) = \lim_m \E(\nu^{Z_m})$ is finite for $\nu\in [0,\nu_0]$, and furthermore that $g_p(\nu)$ must equal $\E(\nu^Z)$.

Now for the above $\nu_0$ we pick $\lambda\in (1,\sqrt{\nu_0})$, and then 
$\rho_0$ such that $1+C/\rho_0 < \lambda$.
Then there exists $n_0$ such that for all $\rho \geq \rho_0$
$$n (1+C/\rho)^n \leq n(1+C/\rho_0)^n \leq \lambda^n, \ \forall n\geq n_0,$$ 
and so we finally have that for all $\rho \geq \rho_0$ 
$$
\E(\tau_0'' \cdot (1+C/\rho)^{\tau_0''}) \leq n_0 \lambda^{n_0} + \E(\lambda^{\tau_0''})
= \lambda^{n_0} + G_p(\lambda^2)< \iy.
$$

\section{Concluding remarks}
The following remarks lead to natural open questions.

\begin{enumerate}
\item 
Suppose that $\rho=\iy$.
In the case of Interacting Urn Models with more than two urns, one can easily derive the law of the number of non-conformist urns (that is the number of urns that do not fixate on the majority color, see \cite{launay1}) when the number $U$ of urns is odd. 
Indeed, if $U$ is an odd number then the majority color in the $U$ urns combined  is necessarily decided at the very first step. Then all the urns that have chosen the majority color immediately fixate and and all the urns that chose the minority color are in configuration $C_2(1)$, and each eventually fixates on the majority color with probability
$$r_1(p)=\frac{p(1+p)}{(1-p)(2-p)},$$ independently of the others.

If the number $U$ of urns is even and strictly bigger than two, then we can still write the probability of fixation as the countable sum of all the urn paths until the majority color is decided but it seems more difficult to figure out an explicit formula for this expression (as the one given by Theorem \ref{theorem}).

\item 
Suppose that $\rho=\iy$.
If there are more than two colors and $U=2$, then generalizing our results is trivial. 
Indeed, if there are $C>2$ colors to start with, then the two urns draw the same color at the first step with probability $1/C$ and they draw two different colors with the remaining probability $(C-1)/C$. In the former case the model starts from configuration $C_3$ and the two urns fixate on the chosen color. In the latter case, the model starts in configuration $C_1(1)$ and the two urns fixate on the same color with probability $q_1(p)$.
If there are more than two urns, then the we cannot conclude much from the results of this paper.
\item
Theorem \ref{T:approx} can also be generalized for $U\geq 2$ and $C\geq 2$. In fact, the proof that $\lim_{\rho\to\infty}\P_\rho[A]=\P_\infty[A]$ is exactly the same (recall that $A$ is the set of all asymptotically possible urn paths). The proof that $\lim_{\rho\to\infty}\P_{\rho}[\bar{A}]=0$ needs to be adapted but can be done in the same fashion. The idea is that if an urn path is in $\bar{A}$ then it has an AI-draw, and its probability is at least $\rho$ times smaller than the same urn path without this AI-draw. Thus $\P_{\rho}[\bar{A}]=O(\rho^{-1})$ which proves the limit.
\end{enumerate}

\bibliographystyle{plain}
\bibliography{biblio}

\end{document}